\documentclass[a4paper,11pt]{amsart}
\usepackage{amsmath,amsthm,amssymb}
\usepackage{mathtools}
\usepackage{fullpage}
\usepackage{hyperref}

\newcommand{\Supp}{\text{Supp}}
\newcommand{\Aut}{\text{Aut}}
\newcommand{\Ord}{\text{Ord}}
\newcommand{\Orb}{\text{Orb}}
\newcommand{\Fix}{F}

\newcommand{\indinv}[2]{i\left(#1; #2\right)}
\newcommand{\indinvL}[2]{i^+\left(#1; #2\right)}
\newcommand{\indinvS}[2]{i^0\left(#1; #2\right)}
\newcommand{\indcard}[2]{i\left(#1, #2\right)}
\newcommand{\indcardleq}[2]{\indcard{#1}{\le#2}}
\newcommand{\ind}[1]{i\left(#1\right)}

\newcommand{\defeq}{\coloneqq}
\newcommand{\e}{\text{e}}
\newcommand{\id}{\text{id}}
\newcommand{\symm}{\triangle}

\newcommand{\BB}{\mathcal{B}}
\newcommand{\IIL}[2]{\mathcal{I}^+\!\left(#1; #2\right)}
\newcommand{\IIS}[2]{\mathcal{I}^0\!\left(#1; #2\right)}
\newcommand{\II}[1]{\mathcal{I}\left(#1\right)}
\newcommand{\IIinv}[2]{\II{#1;#2}}
\newcommand{\KK}{\mathcal{K}}
\newcommand{\MM}{\mathbb{M}}
\newcommand{\PP}{\mathcal{P}}
\renewcommand{\SS}{\mathbb{S}}
\newcommand{\TT}{\mathbb{T}}

\newcommand{\ignore}[1]{}

\newtheorem{theorem}{Theorem}[section]
\newtheorem{lemma}[theorem]{Lemma}
\newtheorem{conjecture}[theorem]{Conjecture}
\newtheorem{corollary}[theorem]{Corollary}

\title{Asymptotics of symmetry in matroids}
\author{Rudi Pendavingh and Jorn van der Pol}
\thanks{This research was supported by the Netherlands Organisation for Scientific Research (NWO) grant 623.001.211.}
\address{\normalfont Eindhoven University of Technology, Eindhoven, The Netherlands. {\em Contact:} \{R.A.Pendavingh, J.G.v.d.Pol\}@tue.nl.}
\begin{document}

\maketitle

\begin{abstract}
	We prove that asymptotically almost all matroids have a trivial automorphism group, or an automorphism group generated by a single transposition. Additionally, we show that asymptotically almost all sparse paving matroids have a trivial automorphism group.
\end{abstract}

\section{Introduction}

Let~$M$ be a matroid. An {\em automorphism} of~$M$ is a permutation~$\pi$ of its ground set, such~$X$ is a basis if and only if~$\pi(X)$ for all subsets~$X$ of the ground set. The automorphisms of a matroid form a group under composition, the {\em automorphism group}, for which we write~$\Aut(M)$.

$\Aut(M)$ always contains the identity permutation, $\id$, which maps any element to itself. If~$\Aut(M) = \{\id\}$, the automorphism group is called {\em trivial}. If this is the case, the matroid~$M$ is called asymmetric, and otherwise it is called symmetric.

Let~$\PP$ be a matroid property, and consider the fraction of matroids satisfying~$\PP$,
\begin{equation*}
	\frac{|\{M \in \MM(n) : \text{$M$ has property~$\PP$}\}|}{|\MM(n)|},
\end{equation*}
where~$\MM(n)$ denotes the set of matroids with ground set~$E=\{1,2,\ldots, n\}$. If this fraction tends to~$1$ as~$n\to\infty$, then we say that~$\PP$ holds for {\em asymptotically almost all} matroids.

\begin{conjecture}\label{conj:asymmetric}
	Asymptotically almost all matroids are asymmetric.
\end{conjecture}

The conjecture is due to Mayhew, Newman, Welsh, and Whittle~\cite{MayhewNewmanWelshWhittle2011}, who conjecture that asymptotically almost all unlabelled matroids are asymmetric, and show that this is equivalent to Conjecture~\ref{conj:asymmetric}. In this paper, we present two partial resolutions to this conjecture.

\begin{theorem}\label{thm:main-matroids}
	Asymptotically almost all matroids are asymmetric or have an automorphism group that is generated by a transposition.
\end{theorem}

Write~$\SS(n)$ for the set of sparse paving matroids with ground set~$E=\{1,2,\ldots,n\}$. Analogous to the matroid case, we say that asymptotically almost all sparse paving matroids have property~$\PP$ if the fraction
\begin{equation*}
	\frac{|\{M \in \SS(n) : \text{$M$ has property~$\PP$}\}|}{|\SS(n)|}
\end{equation*}
converges to~$1$ as~$n\to\infty$.

\begin{theorem}\label{thm:main-sparsepaving}
	Asymptotically almost all sparse paving matroids are asymmetric.
\end{theorem}

\section{Preliminaries}

\subsection{Matroid enumeration}

In this section, we collect some results on matroid enumeration that we require.

We write~$m(n)$ (resp.\ $s(n)$) for the number of matroids (resp.\ sparse paving matroids) on ground set~$[n]$. It is clear that~$s(n) \le m(n)$. A lower bound on the number of sparse paving matroids follows from~\cite[Theorem~1]{GrahamSloane1980},
\begin{equation}\label{eq:matroids-lowerbound}
	\log s(n) \ge \frac{1}{n} \binom{n}{\lfloor n/2\rfloor}.
\end{equation}
An upper bound that matches the lower bound up to a constant factor was obtained in~\cite{BPvdP2015},
\begin{equation}\label{eq:matroids-upperbound}
	\log m(n) \le \frac{2+o(1)}{n} \binom{n}{\lfloor n/2\rfloor}.
\end{equation}


In addition to~$\MM(n)$ and~$\SS(n)$, we write~$\MM(n,r)$ and~$\SS(n,r)$ for the set of matroids (resp.\ sparse paving matroids) on ground set~$[n]$ of rank~$r$, and we will write~$m(n,r)$ and~$s(n,r)$ for their respective cardinalities.

Every matroid~$M\in\MM(n,r)$ is determined by its set of bases, or equivalently its set of non-bases; the non-bases of a matroid of rank~$r$ are the dependent $r$-subsets of the ground set.
We distinguish two types of non-bases: those that are circuit-hyperplanes, and those that are not. Following the notation introduced in~\cite{PvdP2016b}, we will write~$W(M)$ for the set of circuit-hyperplanes of a matroid~$M$, and~$U(M)$ for the set of non-bases that are not circuit-hyperplanes. As~$M$ is determined by its non-bases, it is also determined by the pair~$(U(M), W(M))$.

In~\cite[Theorem~1.3]{PvdP2016b}, it is shown that there is a sequence of sets~$\mathcal{U}_n$ whose cardinality grows relatively slowly, such that~$U(M) \in \mathcal{U}_n$ for all but a vanishing fraction of matroids in~$\MM(n)$. This implies that if $W(M)$ takes its value in a small set for all matroids~$M$ as well, then the class is necessarily small. This is made more precise in the following theorem, which is a variant of ~\cite[Corollary~1.4]{PvdP2016b} that is obtained by replacing any occurrence of~$\log m(n)$ by~$\log s(n)$.

\begin{theorem}\label{thm:thin-sufficient}
	Let~$\mathcal{M}$ be a class of matroids, and let~$\varepsilon > 0$. If
	\begin{equation*}
		\log |\{W(M) : M \in \mathcal{M} \cap \MM(n,r)\}| \le (1-\varepsilon) \log s(n)
	\end{equation*}
	for all~$0 \le r \le n$, then~$|\mathcal{M} \cap \MM(n)| = o(s(n))$.
\end{theorem}

The following is a detailed version of~\cite[Lemma~5.2]{PvdP2016b}.

\begin{theorem}\label{thm:ind-small}
	Let~$f(\kappa) = \kappa \log(2\e/\kappa)$. For all~$\kappa > 0$, as~$n \to \infty$,
	\begin{equation*}
		\max_{0 \le r \le n} \log \indcardleq{J(n,r)}{\kappa \log s(n)} \le (f(\kappa) + o(1)) \log s(n).
	\end{equation*}
\end{theorem}
\cite[Lemma~5.2]{PvdP2016b} can be recovered by setting~$\kappa = 1/5$, and noting that~$f(1/5) < 1$. The function~$f(\kappa)$ has the property that~$\lim_{\kappa \downarrow 0} f(\kappa) = 0$. We note, in particular, that~$f(1/13) < 0.48$.

The following lemmata show that asymptotically almost all (sparse paving) matroids have rank close to half the number of elements of the ground set.

\begin{lemma}[{\cite[Theorem~16]{PvdP2015}}]\label{lemma:rank-general}
	Let~$\beta > \sqrt{\frac{\ln 2}{2}}$. Asymptotically almost all matroids have rank in the interval~$[n/2 - \beta\sqrt{n}, n/2 + \beta\sqrt{n}]$.
\end{lemma}

\begin{lemma}\label{lemma:rank}
	Let~$\beta > \sqrt{\frac{\ln 2}{2}}$. Asymptotically almost all sparse paving matroids have rank in the interval~$[n/2-\beta\sqrt{n}, n/2+\beta\sqrt{n}]$.
\end{lemma}

\subsection{Stable sets in graphs}

Let~$G$ be a graph with vertex set~$V$. A set~$I \subseteq V$ is called a {\em stable set} if no two distinct vertices in~$I$ are adjacent in~$G$. We write~$\II{G}$ for the set of all stable sets in~$G$, and~$\ind{G}\defeq|\II{G}|$ for its cardinality. Similarly, we write~$\indcard{G}{k}$ (resp.\ $\indcard{G}{\le k}$) for the stable sets of cardinality~$k$ (resp.\ cardinality at most~$k$).

If~$\pi\colon V\to V$ is a bijection mapping vertices to vertices and~$I$ is a stable set, then~$I$ is called~$\pi$-invariant if~$\pi(I) = I$. We write~$\indinv{G}{\pi}$ for the number of $\pi$-invariant stable sets in~$G$.

\subsection{Group theory}

We write~$S_n$ for the symmetric group on~$[n]$, and~$\id$ for the identity element in this group. We will use the Greek letter~$\pi$ to refer to elements in~$S_n$.

Given~$\pi \in S_n$, its {\em order}, written~$\Ord(\pi)$, is defined as the smallest positive integer~$k$ such that~$\pi^k = \id$, and we write $\Supp(\pi) \defeq \{e \in [n] : \pi(e) \neq e\}$ for its {\em support}.

A permutation~$\pi$ is a {\em cycle} if there exists a subset~$\{e_1, e_2, \ldots, e_k\} \subseteq [n]$ such that~$\pi(e_i) = e_{i+1}$ for~$i=1,2,\ldots,k-1$, and~$\pi(e_k) = e_1$, while~$\pi$ fixes every other element. If this is the case, we write~$\pi = (e_1, e_2, \ldots, e_k)$. It is clear that~$\Supp(\pi) = \{e_1, e_2, \ldots, e_k\}$ in this case.

Every permutation~$\pi$ admits a representation as a product of disjoint cycles, i.e.\ $\pi = \gamma_1\gamma_2\ldots\gamma_M$, with~$\gamma_i$ a cycle for each~$i$, and $\Supp(\gamma_i)\cap\Supp(\gamma_j) = \emptyset$ whenever~$i\neq j$. This is called the {\em disjoint cycle notation}. We will always suppress cycles of length~$1$ in the disjoint cycle notation, which implies that the disjoint cycle notation is unique up to reordering the factors.

The group~$S_n$ acts pointwise on~$\binom{[n]}{r}$; for~$X \in \binom{[n]}{r}$, we write~$\pi(X)\defeq \{\pi(x) : x \in X\}$. If~$\pi \in S_n$ and~$X \in \binom{[n]}{r}$, we write~$\Orb_\pi(X)\defeq\{X, \pi(X), \pi^2(X), \ldots\}$, when the permutation is clear from the context we will suppress the subscript~$\pi$. Clearly, the cardinality of~$\Orb(X)$ is at most the order of~$\pi$, and in fact~$|\Orb_\pi(X)|$ always divides~$\Ord(\pi)$.

Note that~$\Orb_\pi(X)$ is a singleton if and only if~$\pi(X) = X$.

A subset~$\mathcal{X} \subseteq \binom{[n]}{r}$ is called {\em $\pi$-invariant} if~$\pi(X) \in \mathcal{X}$ for all~$X \in \mathcal{X}$. This is the case precisely when~$\mathcal{X}$ is the union of $\pi$-orbits.

\subsection{The Johnson graph}

If~$E$ is a finite set, and~$0 \le r \le |E|$, then we write
\begin{equation*}
	\binom{E}{r} \defeq \{ X \subseteq E : |X|=r\}.
\end{equation*}

The {\em Johnson graph}~$J(E,r)$ is the graph on vertex set~$\binom{E}{r}$, in which two vertices~$X, Y \in \binom{E}{r}$ are adjacent if and only if~$|X\symm Y| = 2$. We abbreviate~$J(n,r) \defeq J([n],r)$.

The Johnson graphs are relevant in this paper because of the following observation.

\begin{lemma}
	If~$M \in \MM(n,r)$, then~$W(M) \in \II{J(n,r)}$.
\end{lemma}

Every permutation~$\pi \in S_n$ gives rise to an automorphism of the Johnson graph~$J(n,r)$. (In fact, these permutations form the complete automorphism group, except when~$n = 2r$, in which case there is one extra automorphism, namely the function that maps any vertex to its complement in~$[n]$ (see e.g.\ \cite[Theorem~9.2.1]{BrouwerCohenNeumaier1989}).)

\subsection{Binomial coefficients}

The binomial coefficient~$\binom{n}{\lfloor n/2\rfloor}$ is known as the {\em central binomial coefficient}. Based on Stirling's approximation to the factorial function, asymptotically tight bounds for the central binomial theorem can be computed; we will use
\begin{equation}\label{eq:central-binomial-coefficient}
	\sqrt{2/\pi} \frac{2^n}{\sqrt{n}} \left(1-\frac{1}{n}\right) \le \binom{n}{\lfloor n/2\rfloor} \le \sqrt{2/\pi} \frac{2^n}{\sqrt{n}}.
\end{equation}
Using~\eqref{eq:central-binomial-coefficient}, different central binomial coefficients can be compared; we will require the following inequality:
\begin{equation}\label{eq:central-binomial-coefficient-compare}
	\binom{n-m}{\left\lfloor\frac{n-m}{2}\right\rfloor} \le \frac{n}{n-1} \sqrt{\frac{n}{n-m}} 2^{-m} \binom{n}{\lfloor n/2\rfloor}.
\end{equation}

While~\eqref{eq:central-binomial-coefficient} gives precise asymptotics for the central binomial coefficient, we will also require bounds for binomial coefficients that are close to the central binomial coefficient.

\begin{lemma}[{\cite[Equation~(5.41)]{SpencerFlorescu2014}}]\label{lemma:central-binomial-coefficient-deviation}
	If~$k = o\left(n^{2/3}\right)$, then~$\binom{n}{\lfloor n/2\rfloor + k} = (1+o(1)) \sqrt{\frac{2}{\pi}} \e^{-2k^2/n} \frac{2^n}{\sqrt{n}}$.
\end{lemma}

\section{Proofs}

\subsection{Outline of the proofs}

Theorem~\ref{thm:main-matroids} and Theorem~\ref{thm:main-sparsepaving} are proved in this section. As it turns out, the automorphism groups that are generated by a single transposition, which occur in Theorem~\ref{thm:main-sparsepaving}, require a different approach than the other non-trivial groups. This will be reflected in the structure of this section.

Throughout the section, $\pi$ will always be an element of~$S_n$. We write~$\MM(n,r;\pi)$ for the collection of rank-$r$ matroids on ground set~$[n]$ that have~$\pi$ as an  automorphism, i.e.\
\begin{equation*}
	\MM(n,r;\pi) \defeq \left\{ M \in \MM(n,r) : \pi \in \Aut(M)\right\}.
\end{equation*}
Moreover, for any subset~$\Sigma \subseteq S_n$, we define~$\MM(n,r; \Sigma) \defeq \bigcup_{\pi \in \Sigma} \MM(n,r;\pi)$. In addition, we define~$\MM(n; \pi) \defeq \bigcup_{r=0}^n \MM(n,r; \pi)$, and~$\MM(n;\Sigma) \defeq \bigcup_{r=0}^n \MM(n; \pi)$. We use lower case letters to denote cardinalities, e.g.\ $m(n,r;\pi) \defeq | \MM(n,r;\pi)|$, and so on.

Analogously, we write~$\SS(n,r; \pi)$ for the set of those sparse paving matroids in~$\SS(n,r)$ that have~$\pi$ as an automorphism, $s(n,r; \pi)$ for its cardinality, and so on.

In this paper, two sets of permutations play a prominent role. These are
\begin{equation*}
	\Sigma_{\ge 3} \defeq \{ \pi \in S_n : |\Supp(\pi)| \ge 3\}, \qquad\text{and}\qquad \Sigma_2 \defeq \{\pi \in S_n : |\Supp(\pi)| = 2\}.
\end{equation*}
Note that~$\Sigma_2$ is the set of transpositions in~$S_n$.

In this section, we will prove the following two results.
\begin{theorem}\label{thm:matroids}
	$\lim_{n\to\infty} \frac{m(n; \Sigma_{\ge 3})}{s(n)} = 0$.
\end{theorem}
\begin{theorem}\label{thm:sparsepaving}
	$\lim_{n \to \infty} \frac{s(n; \Sigma_2)}{s(n)} = 0$.
\end{theorem}

It is easily verified that these two theorems imply the main results, Theorem~\ref{thm:main-matroids} and Theorem~\ref{thm:main-sparsepaving}. As every asymmetric matroid that does not have a transposition as automorphism is in~$\MM(n; \Sigma_{\ge 3})$, and~$m(n) \ge s(n)$, Theorem~\ref{thm:matroids} implies Theorem~\ref{thm:main-matroids}. Similarly, as the number of symmetric sparse paving matroids is at most~$s(n; \Sigma_2) + s(n; \Sigma_{\ge 3})$, and $s(n; \Sigma_{\ge 3}) \le m(n; \Sigma_{\ge 3})$, Theorem~\ref{thm:main-sparsepaving} follows upon combining Theorem~\ref{thm:matroids} and~\ref{thm:sparsepaving}.


\subsection{Permutations that move at least three elements}

A central role in the proof of Theorem~\ref{thm:matroids} will be played by the circuit-hyperplanes of matroids whose automorpism group contains a given permutation~$\pi$.

Let~$M \in \MM(n,r)$, and~$\pi \in S_n$. Observe that if~$\pi \in \Aut(M)$, then~$W(M)$ is a $\pi$-invariant stable set in~$J(n,r)$. We show that if $\pi\in\Sigma_{\ge 3}$, then the number of $\pi$-invariant stable sets in~$J(n,r)$ is small---so small in fact, that even after summing over all~$\pi \in \Sigma_{\ge 3}$, the resulting bound on~$|\{W(M) : M \in \MM(n,r;\Sigma_{\ge 3})\}|$ is sufficiently small for an application of Theorem~\ref{thm:thin-sufficient}, which then implies Theorem~\ref{thm:matroids}.

For a permutation~$\pi \in S_n$, define
\begin{equation*}
	\Fix(\pi) \defeq \left\{X \in \binom{[n]}{r} : \pi(X) = X\right\}
\end{equation*}
for the set of $r$-sets that are fixed under~$\pi$. Recall that~$\IIinv{J(n,r)}{\pi}$ is the collection of all $\pi$-invariant stable sets in~$J(n,r)$; we identity two special subsets of~$\IIinv{J(n,r)}{\pi}$, namely
\begin{equation*}
	\begin{split}
		\IIS{J(n,r)}{\pi} &\defeq \left\{I \in \IIinv{J(n,r)}{\pi} : I \subseteq \Fix(\pi)\right\},\qquad\text{and} \\
		\IIL{J(n,r)}{\pi} &\defeq \left\{I \in \IIinv{J(n,r)}{\pi} : I \cap \Fix(\pi) = \emptyset\right\}.
	\end{split}
\end{equation*}
These sets do not form a bipartition of~$\IIinv{J(n,r)}{\pi}$. Rather, they form a ``basis'' in the sense that each~$I \in \IIinv{J(n,r)}{\pi}$ can be written as the disjoint union~$I = I^0 \cup I^+$, where~$I^0 \defeq I \cap F(\pi) \in \IIS{J(n,r)}{\pi}$, and~$I^+ \defeq I\setminus F(\pi) \in \IIL{J(n,r)}{\pi}$.

We use lower case letters to denote cardinality, so
\begin{equation*}
	\indinvS{J(n,r)}{\pi} \defeq |\IIS{J(n,r)}{\pi}|, \qquad\text{and}\qquad \indinvL{J(n,r)}{\pi} \defeq |\IIL{J(n,r)}{\pi}|.
\end{equation*}

The following lemma bounds~$\indinvS{J(n,r)}{\pi}$ in terms of stable sets in smaller Johnson graphs.

\begin{lemma}\label{lemma:indinv-sum}
	For all~$0 \le r \le n$, if~$\pi \in S_n$ has a decomposition into~$M$ disjoint cycles, $\pi=\gamma_1\gamma_2\ldots\gamma_M$, in which~$\gamma_j$ has length~$\ell_j = |\Supp(\gamma_j)|$, then
	\begin{equation*}
		\log \indinvS{J(n,r)}{\pi}
		\le 2^M \log s(n-m),
	\end{equation*}
	where~$m = |\Supp(\pi)|$.
\end{lemma}

\begin{proof}
	
	Let~$\pi$ be as in the statement of the lemma. If~$X \in \Fix(\pi)$, then for each~$j \in [M]$ either~$\Supp(\gamma_j) \cap X =\emptyset$, or~$\Supp(\gamma_j) \subseteq X$. Let
	\begin{equation*}
		P_\mathcal{J} \defeq \left\{X \in \binom{[n]}{r} : X \cap \Supp(\pi) = \bigcup_{j \in \mathcal{J}} \Supp(\gamma_j)\right\}.
	\end{equation*}
	The subgraph of~$J(n,r)$ induced by~$P_\mathcal{J}$ is isomorphic to~$J(n-m,r')$, where~$r' = r - \sum_{j \in \mathcal{J}} \ell_j$.
	
	If~$X \in F(\pi)$, then there exists a unique~$\mathcal{J} \subseteq [M]$ such that~$X \in P_\mathcal{J}$. It follows that if~$I \in \IIS{J(n,r)}{\pi}$, then~$\{I \cap P_\mathcal{J} : \mathcal{J} \subseteq [M]\}$. Moreover, each~$I \cap P_\mathcal{J}$ is a stable set in~$J(n,r)[P_\mathcal{J}]$. Thus,
	\begin{equation*}
		\log \indinvS{J(n,r)}{\pi} \le \sum_{\mathcal{J} \subseteq [M]} \log \ind{J\left(n-m, r - \sum_{j \in \mathcal{J}} \ell_j\right)}.
	\end{equation*}
	The lemma now follows since~$\ind{J(n-m, r')} \le s(n-m)$ for all~$r'$.
\end{proof}

\begin{lemma}\label{lemma:indinvS-small}
	$\max\limits_{\substack{\pi \in \Sigma_{\ge 3} \\ 0 \le r \le n}} \log \indinvS{J(n,r)}{\pi} \le \left(1/2 + o(1)\right) \log s(n)$ as $n\to\infty$.
\end{lemma}

\begin{proof}
	By Lemma~\ref{lemma:indinv-sum}, $\max_{0 \le r \le n} \log \indinvS{J(n,r)}{\pi} \le 2^M \log s(n-m)$ for all~$\pi \in S_n$, where~$m = |\Supp(\pi)|$ and~$M$ is the number of cycles in the disjoint cycle representation of~$\pi$. As~$M \le \lfloor m/2\rfloor$, it follows that
	\begin{equation}\label{eq:indinvS-small-1}
		\max_{\substack{\pi \in S_{\ge 3} \\ 0 \le r \le n}} \log \indinvS{J(n,r)}{\pi} \le \max_{3 \le m \le n} 2^{\lfloor m/2\rfloor} \log s(n-m).
	\end{equation}
	It remains to bound the right-hand side of~\eqref{eq:indinvS-small-1}.
	
	We have
	\begin{equation}\label{eq:indinvS-small-2}
		\begin{aligned}
			\max_{\left\lceil\frac{2n}{3}\right\rceil \le m \le n} 2^{\lfloor m/2\rfloor} \log s(n-m)
				& \le 2^{\lfloor n/2\rfloor} \log s(\lfloor n/3\rfloor)
					&& \\
				& \le 2^{\lfloor n/2\rfloor} \frac{6+o(1)}{n} \binom{\lfloor n/3\rfloor}{\lfloor n/6\rfloor}
					&& \text{by~\eqref{eq:matroids-upperbound}} \\
				& \le \frac{18+o(1)}{n} \binom{n}{\lfloor n/2\rfloor} 2^{-n/6}
					&& \text{by~\eqref{eq:central-binomial-coefficient-compare}} \\
				&= o(\log s(n))
					&& \text{by~\eqref{eq:matroids-lowerbound}.}
		\end{aligned}
	\end{equation}
	
	Next, suppose that~$3 \le m \le  \left\lfloor\frac{2n}{3}\right\rfloor$. As~$n-m \to\infty$, an application of~\eqref{eq:matroids-upperbound} shows that	
	\begin{equation*}
		2^{\lfloor m/2\rfloor} s(n-m)
			\le 2^{\lfloor m/2\rfloor} \frac{2+o(1)}{n-m} \binom{n-m}{\left\lfloor\frac{n-m}{2}\right\rfloor},
	\end{equation*}
	which, by~\eqref{eq:central-binomial-coefficient-compare}, is at most
	\begin{equation*}
		2^{\lfloor m/2\rfloor} \frac{2+o(1)}{n-m} \sqrt{\frac{n}{n-m}} 2^{-m} \binom{n}{\lfloor n/2\rfloor}
			\le 2^{-\lceil m/2\rceil} (2+o(1)) \left(\frac{n}{n-m}\right)^{3/2} \log s(n),
	\end{equation*}
	so that
	\begin{equation}\label{eq:indinvS-small-3}
		\max_{3 \le m \le \left\lfloor\frac{2n}{3}\right\rfloor} 2^{\lfloor m/2\rfloor} s(n-m) \le \left(1/2 + o(1)\right) \log s(n).
	\end{equation}
	Combining~\eqref{eq:indinvS-small-2} and~\eqref{eq:indinvS-small-3} with~\eqref{eq:indinvS-small-1} proves the lemma.
\end{proof}


Observe that if~$I$ is a $\pi$-invariant stable set, and~$I' \subseteq I$ contains at least one vertex from each $\pi$-orbit that is contained in~$I$, then~$I$ can be reconstructed by closing~$I'$ under $\pi$-images. This observation will be used in the proof of the following lemma.

\begin{lemma}\label{lemma:indinv-small}
	There exists~$\varepsilon > 0$ such that for sufficiently large~$n$ and all~$0 \le r \le n$, if~$\pi \in \Sigma_{\ge 3}$, then~$\log \indinv{J(n,r)}{\pi} \le (1-\varepsilon) \log s(n)$.
\end{lemma}

\begin{proof}
	For a $\pi$-invariant stable set~$I$ in~$J(n,r)$, let us write~$\lambda(I)$ for the number of ``large'' orbits that it contains (i.e.\ orbits consisting of at least two vertices).
	
	Define~$\Lambda \defeq \frac{1}{13} \log s(n)$. Call~$I$ ``complex'' if~$\lambda(I) > \Lambda$. Either the majority of $\pi$-invariant stable sets is complex, or non-complex. We will show that~$\indinv{J(n,r)}{\pi}$ is small either way.
	
	
	Let us first show that the lemma holds if the majority of $\pi$-invariant stable sets is complex. Each complex set gives rise to at least~$3^{\lambda(I)} \ge 3^\Lambda$ stable sets, since we can take any non-empty subset from each large orbit. By the previous paragraph, $I$ can be reconstructed from each such subset. Hence, if at least half of the $\pi$-invariant stable sets is complex, we have
	\begin{equation*}
		\indinv{J(n,r)}{\pi} \le 2 \ind{J(n,r)} 3^{-\tfrac{1}{13} \log s(n)},
	\end{equation*}
	and the lemma follows.
	
	Next, we will show that the lemma holds if the majority of $\pi$-invariant stable sets is non-complex. Recall that each $\pi$-invariant stable set~$I$ can be written as the disjoint union of~$I^0 \in \IIS{J(n,r)}{\pi}$ and~$I^+ \in \IIL{J(n,r)}{\pi}$. We bound the number of~$I^0$ and~$I^+$ associated with non-complex~$I$ in this way separately.
	
	Note that~$I^+$ can be reconstructed from a stable set of size~$\lambda(I^+) = \lambda(I) \le \Lambda$, by restricting~$I^+$ to a set containing a single vertex from each of its orbits. Thus, the number of possible~$I^+$ is at most~$\indcardleq{J(n,r)}{\Lambda}$, which can be bounded by Theorem~\ref{thm:ind-small}. We obtain that, for sufficiently large~$n$, the logarithm of the number of possible~$I^+$ is at most
	\begin{equation}\label{eq:non-complex-L}
		\log \indcardleq{J(n,r)}{\Lambda} \le 0.48 \log s(n).
	\end{equation}
	
	
	An application of Lemma~\ref{lemma:indinvS-small} shows that for sufficiently large~$n$,
	\begin{equation}\label{eq:non-complex-S}
		\log \indinvS{J(n,r)}{\pi} \le 0.51 \log s(n).
	\end{equation}
	
	Suppose that at least half of the $\pi$-invariant stable sets is non-complex, i.e.\ $\lambda(I) \le \Lambda$. Combining~\eqref{eq:non-complex-L} and~\eqref{eq:non-complex-S} shows that
	\begin{equation*}
		\log \indinv{J(n,r)}{\pi} \le 1 + 0.48 \log s(n) + 0.51 \log s(n),
	\end{equation*}
	which proves the lemma.
\end{proof}

\begin{proof}[Proof of Theorem~\ref{thm:matroids}]
	As~$|\{W(M) : M \in \MM(n,r; \pi)\}| = \indinv{J(n,r)}{\pi}$, it follows that
	\begin{equation*}
		|\{W(M) : M \in \MM(n,r;\Sigma_{\ge 3})\}| \le \sum_{\pi \in \Sigma_{\ge 3}'} \indinv{J(n,r)}{\pi}.
	\end{equation*}
	Note that~$|\Sigma_{\ge 3}| < n! \le n^n$ so by an application of Lemma~\ref{lemma:indinv-small}, there is~$\varepsilon > 0$ such that, for sufficiently large~$n$,
	\begin{equation*}
		\log |\{W(M) : M \in \MM(n,r; \Sigma_{\ge 3})\}| \le (1-\varepsilon) \log s(n) + n \log n \le (1-\varepsilon/2) \log s(n)
	\end{equation*}
	for all~$0 \le r \le n$. Theorem~\ref{thm:matroids} thus follows from an application of Theorem~\ref{thm:thin-sufficient}.
\end{proof}

\subsection{Transpositions}\label{ss:transpositions}

Let~$\pi = (e,f) \in \Sigma_2$ be a transposition. Recall that ($\pi$-invariant) sparse paving matroids of rank~$r$ on groundset~$[n]$ are in one-to-one correspondence with ($\pi$-invariant) stable sets in~$J(n,r)$. The main step in the proof of Theorem~\ref{thm:sparsepaving} is showing that we can associate to any $\pi$-invariant stable set in~$J(n,r)$ a large family of stable sets that are not $\pi$-invariant.

The transposition~$\pi$ partitions the vertices of~$J(n,r)$ into four classes, based on the intersection with the set~$\{e,f\}$. Let us write~$V_\emptyset, V_e, V_f, V_{e,f}$ for the vertices in~$J(n,r)$ corresponding to the subscript, and write~$J(n,r)_\xi \defeq J(n,r)[V_\xi]$ for the corresponding induced subgraph.

Each of these graphs is isomorphic to a Johnson graph with smaller parameters, to wit
\begin{equation*}
	J(n,r)_\emptyset \cong J(n-2,r), \quad J(n,r)_e \cong J(n,r)_f \cong J(n-2,r-1),\quad\text{and}\quad J(n,r)_{e,f} \cong J(n-2, r-2).
\end{equation*}

Moreover, there is precisely a matching between the vertices in~$V_e$ and those in~$V_f$. It follows that~$J(n,r)[V_e \cup V_f] \cong J(n-2,r-1) \Box K_2$, the Cartesian product of~$J(n-2,r-1)$ and~$K_2$.

Each $\pi$-invariant stable set is contained in~$V_\emptyset \cup V_{e,f}$, for if~$X \in V_e \cup V_f$ would be in the stable set, then so would~$\pi(X) = X \symm \{e,f\}$. However, $X$ is adjacent to~$X \symm \{e,f\}$, thus contradicting stability.

In fact, not only is every $\pi$-invariant stable set contained in~$V_\emptyset \cup V_{e,f}$, but every $\pi$-invariant stable set in~$J(n,r)$ can be constructed by combining a stable set in~$V_\emptyset$ and a stable set in~$V_{e,f}$. In particular, this means that
\begin{equation*}
	\indinv{J(n,r)}{\pi} = \ind{J(n-2,r-2)} \times \ind{J(n-2,r)}.
\end{equation*}

Clearly~$\indinv{J(n,r)}{\pi} \le \ind{J(n,r)}$. The following lemma gives a family of related bounds.
\begin{lemma}\label{lemma:indinv-kbound}
	For all~$k \ge 0$,
	\begin{equation*}
		\indinv{J(n,r)}{\pi} \le \frac{(r(n-r))^k}{\indcard{J(n-2,r-1)\Box K_2}{k}} \ind{J(n,r)}.
	\end{equation*}
\end{lemma}

\begin{proof}
	We will prove the lemma by counting in two ways the number of pairs~$(I,A)$, where~$I$ is a $\pi$-invariant stable set in~$J(n,r)$, and~$A$ is a stable set of cardinality~$k$ in~$J(n,r)[V_e \cup V_f]$.
	
	On the one hand, the number of such pairs is exactly~$\indinv{J(n,r)}{\pi} \times \indcard{J(n-2,r-1)\Box K_2}{k}$.
	
	On the other hand, we show that the number of such pairs is at most~$\ind{J(n,r)} \times (r(n-r))^k$. Together, these two observations prove the lemma.
	
	To prove the second observation, consider the map~$F(I,A) = I\cup\{A\} \setminus N(A)$. Clearly, for each pair~$(I, A)$, $F(I,A)$ is a stable set in~$J(n,r)$. We claim that at most~$(r(n-r))^k$ of the pairs give rise to the same image under~$F$.
	
	Starting from~$F(I,A)$, note that~$A$ is determined by~$A = F(I,A) \cap (V_e \cup V_f)$; here we use that~$I \subseteq V_\emptyset \cup V_{e,f}$, while~$A \subseteq V_e \cup V_f$. It remains to reconstruct~$I \cap N(A)$. A vertex~$X \in V_e \cup V_f$ has exactly~$n-r-1$ neighbours among the vertices in~$V_\emptyset$ (and these vertices form a clique), and it has~$r-1$ neighbours among the vertices in~$V_{e,f}$ (and these form a clique as well). Thus, for each~$X \in A$, $I \cap N(X)$ can take at most~$r(n-r)$ different values. The claim follows by taking the product over all~$X \in A$.
\end{proof}

\begin{proof}[Proof of Theorem~\ref{thm:sparsepaving}]
	We will show that
	\begin{equation}\label{eq:sparsepaving-inequality}
		\sum_{r=0}^n s(n,r; \Sigma_2) = o(s(n)),
	\end{equation}
	which implies Theorem~\ref{thm:sparsepaving}.
	Let~$R \defeq \{0,1\ldots, n\} \cap (n/2-\sqrt{n}, n/2+\sqrt{n})$, and~$R^c \defeq \{0,1,\ldots,n\} \setminus R$.
	We will prove~\eqref{eq:sparsepaving-inequality} by splitting the sum into two parts, corresponding to~$R^c$ and~$R$, respectively, and showing that both parts are~$o(s(n))$.

	From Lemma~\ref{lemma:rank} we know that asymptotically almost all sparse paving matroids have rank in~$R$, and hence
	\begin{equation}\label{eq:sparsepaving-inequality-step1}
		\sum_{r \in R^c} s(n,r;\Sigma_2) \le \sum_{r \in R^c} s(n,r) = o(s(n)).
	\end{equation}

	Next, let~$r \in R$. By Lemma~\ref{lemma:central-binomial-coefficient-deviation}, there is a constant~$c$ such that, for sufficiently large~$n$, ~$\binom{n-2}{r-1} = \frac{r(n-r)}{n(n-1)}\binom{n}{r} \ge c\frac{2^n}{\sqrt{n}}$. Fix any transposition~$\pi \in \Sigma_2$. By Lemma~\ref{lemma:indinv-kbound}, applied here with~$k=1$,
	\begin{equation*}
		s(n,r;\pi) = \indinv{J(n,r)}{\pi} \le \frac{r(n-r)}{2\binom{n-2}{r-1}} \ind{J(n,r)} \le \frac{n^2 \sqrt{n}}{8c 2^n} s(n)
	\end{equation*}
	for all sufficiently large~$n$. As~$|R| \le 2\sqrt{n} + 1$, and~$|\Sigma_2| = \binom{n}{2}$, it follows that
	\begin{equation}\label{eq:sparsepaving-inequality-step2}
		\sum_{r \in R} s(n,r;\Sigma_2) \le \sum_{r\in R} \sum_{\pi \in \Sigma_2} s(n,r;\pi) \le (1+o(1))\frac{n^5}{8c 2^n} s(n).
	\end{equation}
	Combining~\eqref{eq:sparsepaving-inequality-step1} and~\eqref{eq:sparsepaving-inequality-step2} proves~\eqref{eq:sparsepaving-inequality}, and hence Theorem~\ref{thm:sparsepaving}.
\end{proof}

\section{Final remarks}

\subsection{Matroids whose automorphism group is generated by a transposition}

The result in Theorem~\ref{thm:matroids} is not quite sufficient to prove Conjecture~\ref{conj:asymmetric}, as it does not give any information about matroids that have an automorphism group that is generated by a single automorphism. In this section, we will further address this issue.

Throughout this section, $\pi = (e,f)$ will be an arbitrary permutation that exchanges the elements~$e$ and~$f$. We write~$\TT(n; \pi) \defeq \{M \in \MM(n) : \Aut(M) = \langle \pi\rangle\}$, $\TT(n) \defeq \bigcup_{\pi \in \Sigma_2} \TT(n; \pi)$, $t(n,\pi) \defeq |\TT(n;\pi)|$, and $t(n) \defeq |\TT(n)|$.

In view of Theorem~\ref{thm:main-matroids}, the following Conjecture is tantamount to proving Conjecture~\ref{conj:asymmetric}.

\begin{conjecture}\label{conj:transpositions}
	$\lim_{n \to \infty} \frac{t(n)}{m(n)} = 0$.
\end{conjecture}

One might try to prove Conjecture~\ref{conj:transpositions} using variant of the proof of Theorem~\ref{thm:sparsepaving} that is geared towards general matroids, rather than sparse paving matroids.

The proof of Theorem~\ref{thm:sparsepaving} is based on the construction of a large number of sparse paving matroids associated with a given $\pi$-invariant sparse paving matroid. This was obtained by forcing an element from set~$V_e \cup V_f$, as defined in Section~\ref{ss:transpositions} into the set of non-bases of the original matroid. In the sparse paving case, this approach works, since all elements in~$V_e \cup  V_f$ are bases of the matroid, each such element has few neighbours among the non-bases in the original matroid.

The situation for general matroids is more complicated in two ways. First, $V_e \cup V_f$ may contain bases. Second, the collection of non-bases in the neighbourhood of some~$X \in V_e \cup V_f$ that we may want to force in the collection of non-bases of the original matroid may be much more complicated, compared to the situation that the original matroid is sparse paving.

The following lemma shows what might happen if we can avoid both complications. Its proof is analogous to the proof of Lemma~\ref{lemma:indinv-kbound}.
\begin{lemma}
	If for an $f_n$-fraction of matroids in~$\MM(n,r; \pi)$ there exists a stable $K_n$-set~$X \subseteq V_e$ with the property that~$X \cup N(X) \subseteq \BB(M)$, then~$m(n; \pi) \le \frac{2^{K_n}}{f_n} m(n)$.
\end{lemma}
If the functions~$K_n$ and~$f_n$ satisfy~$\frac{2^{K_n}}{f_n} = o\left(1/\binom{n}{2}\right)$, then the lemma implies Conjecture~\ref{conj:transpositions}. In view of Lemma~\ref{lemma:rank-general}, this is true even if the lemma holds only for values of~$r$ satisfying~$n/2 - \sqrt{n} \le r \le n/2 + \sqrt{n}$.

Alternatively, we might consider what happens if Conjecture~\ref{conj:transpositions} does not hold.

\begin{lemma}
	Let~$M \in \MM(n,r;\pi)$. $M$ is uniquely determined by~$M\backslash ef$ and~$M/ef$.
\end{lemma}

\begin{proof}
	If~$M\backslash ef = M/ef$, then~$\{e,f\}$ is dependent or codependent in~$M$. They are a pair of loops (resp.\ coloops) if and only if~$r(M/ef) = r$ (resp.\ $r(M/ef) = r$). The set $\{e,f\}$ is a circuit in~$M$ if and only if~$r(M/ef) = r-1$, in which case they are a cocircuit as well. $M\backslash ef$ can be uniquely extended by two elements that form both a circuit and a cocircuit, so this extension must be~$M$.
	
	It remains to show that the lemma holds if~$M\backslash ef \neq M/ef$. We will do this by reconstructing the set of bases of~$M$, based on the sets of bases of the given minors. By definition, as~$\{e,f\}$ is both independent and coindependent,
	\begin{equation*}
		\{B \in \BB(M) : \{e,f\} \subseteq B\} = \{B \cup \{e,f\} : B \in \BB(M/ef)\},
	\end{equation*}
	while
	\begin{equation*}
		\{B \in \BB(M) : \{e,f\}\cap B = \emptyset\} = \BB(M\backslash ef).
	\end{equation*}
	It remains to reconstruct the set of bases that contain exactly one of~$e,f$, or equivalently, the set of non-bases that contain exactly one of~$e,f$. In fact, since~$\pi$ is an automorphism of~$M$, it suffices to reconstruct the set of non-bases that contain~$e$, but not~$f$. Call this set~$\KK$. We claim that
	\begin{multline}\label{eq:KK}
		\KK = \left\{X \in \binom{[n]}{r} : e \in X, f \not\in X,\text{ and } X - e + g \not \in \BB(M/ef) \text{ for all } g \in E\setminus X\setminus\{f\}\right\} \\
			\bigcup \left\{X \in \binom{[n]}{r} : e \in X, f\not\in X, \text{ and } X - e - h \not\in\BB(M\backslash ef) \text{ for all } h \in X\setminus\{e\}\right\},
	\end{multline}
	which depends only on~$M\backslash ef$ and~$M/ef$. That~\eqref{eq:KK} holds follows from the observation that~$X$ is a non-basis if and only if~$X \symm \{e,f\}$ is; this implies that~$X$ is a non-basis if and only if $X\cup\{f\}$ is contained in a hyperplane, or~$X\setminus\{e\}$ contains a circuit. In the former case, every $r$-subset of~$X\cup\{f\}$ is a non-basis, and in the latter case every $r$-subset containing~$X\setminus\{e\}$ is a non-basis.
\end{proof}

It follows from the lemma that
\begin{equation}\label{eq:t-r-upper}
	t(n,r; \pi) \le m(n-2,r-2)m(n-2,r) + 3;
\end{equation}
summing over~$r$ and taking logarithms, this implies that, for sufficiently large~$n$,
\begin{equation}\label{eq:t-upper}
	\log t(n; \pi) \le 2\log m(n-2).
\end{equation}

\begin{lemma}\label{lemma:t-liminf}
	$\liminf\limits_{n\to\infty} \frac{t(n)}{m(n)} = 0$.
\end{lemma}

\begin{proof}
	We argue by contradiction. If the lemma fails, there exists~$\varepsilon > 0$ such that~$t(n) \ge \varepsilon m(n)$, for all~$n$ sufficiently large. By symmetry, $t(n) = \binom{n}{2}t(n; \pi)$ for any transposition~$\pi$. It follows that~$\log m(n) = (1+o(1)) \log t(n; \pi)$; combining this with~\eqref{eq:t-upper} and~\eqref{eq:matroids-upperbound}, we obtain
	\begin{equation*}
		\frac{1}{n} \binom{n}{\lfloor n/2\rfloor} \le \log m(n) = (1+o(1)) m(n; \pi) \le 2 \log m(n-2) \le \frac{4+o(1)}{n-2} \binom{n-2}{\lfloor (n-2)/2\rfloor}.
	\end{equation*}
	Noting that~$\binom{n-2}{\lfloor(n-2)/2\rfloor} = (1/4+o(1)) \binom{n}{\lfloor n/2\rfloor}$, it follows that
	\begin{equation*}
		\lim_{n \to \infty} \frac{ \log m(n)}{\frac{1}{n} \binom{n}{\lfloor n/2}} = 1, \qquad\text{while}\qquad \lim_{n \to\infty} \frac{\log m(n-2)}{\frac{1}{n-2}\binom{n-2}{\lfloor n/2\rfloor}} = 2.
	\end{equation*}
	These two statements cannot hold simultaneously, hence the lemma follows.
\end{proof}

By Lemma~\ref{lemma:t-liminf}, if the limit~$\lim_{n\to\infty} \frac{t(n)}{m(n)}$ exists, then it must be equal to~0, and this would imply Conjecture~\ref{conj:transpositions}.

The bounds on the number of matroids from~\eqref{eq:matroids-lowerbound} and~\ref{eq:matroids-upperbound} imply that
\begin{equation*}
	\liminf\limits_{n \to\infty} \frac{ \log m(n)}{\frac{1}{n} \binom{n}{\lfloor n/2\rfloor}} \ge 1
	\qquad\text{and}\qquad
	\limsup\limits_{n \to\infty} \frac{\log m(n)}{\frac{1}{n} \binom{n}{\lfloor n/2\rfloor}} \le 2.
\end{equation*}
The following lemma, whose proof is similar to that of Lemma~\ref{lemma:t-liminf}, shows that if Conjecture~\ref{conj:transpositions} fails, these inequalities are actually equalities.

\begin{lemma}\label{lemma:t-limsup}
	If~$\limsup\limits_{n \to \infty} \frac{t(n)}{m(n)} > 0$,
	then $\liminf\limits_{n \to\infty} \frac{\log m(n)}{\frac{1}{n}\binom{n}{\lfloor n/2\rfloor}} = 1$
	and $\limsup\limits_{n \to\infty} \frac{\log m(n)}{\frac{1}{n}\binom{n}{\lfloor n/2\rfloor}} = 2$.
\end{lemma}

The following corollary is simply the contrapositive of Lemma~\ref{lemma:t-limsup}; it gives a sufficient condition for Conjecture~\ref{conj:asymmetric}.
\begin{corollary}\label{cor:limits}
	If $\liminf\limits_{n\to\infty} \frac{\log m(n)}{\frac{1}{n}\binom{n}{\lfloor n/2\rfloor}} > 1$, or $\limsup\limits_{n\to\infty} \frac{\log m(n)}{\frac{1}{n}\binom{n}{\lfloor n/2\rfloor}} < 2$, then Conjecture~\ref{conj:asymmetric} holds.
\end{corollary}

We expect that the antecedent in Corollary~\ref{cor:limits} holds in a strong sense, namely that the limit of~$\frac{\log m(n)}{\frac{1}{n} \binom{n}{\lfloor n/2\rfloor}}$ exists.

\subsection{Related conjectures}

Mayhew, Newman, Welsh, and Whittle~\cite{MayhewNewmanWelshWhittle2011} present a number of conjectures that are related to Conjecture~\ref{conj:asymmetric}.
\begin{conjecture}[{\cite[Conjecture~1.6]{MayhewNewmanWelshWhittle2011}}]\label{conj:paving}
	Asymptotically almost all matroids are paving.
\end{conjecture}

By duality, if almost all matroids are paving, then almost all matroids are sparse paving. This observation, combined with Theorem~\ref{thm:main-sparsepaving}, shows that Conjecture~\ref{conj:paving} immediately implies Conjecture~\ref{conj:asymmetric}.

\begin{conjecture}[{\cite[Conjecture~1.10]{MayhewNewmanWelshWhittle2011}}]\label{conj:rank}
	Asymptotically almost all matroids satisfy~$\frac{n-1}{2} \le r \le \frac{n+1}{2}$.
\end{conjecture}

Let~$\tilde{m}(n)$ be the number of matroids on ground set~$[n]$ with rank less than~$\frac{n-1}{2}$ or larger than~$\frac{n+1}{2}$. Conjecture~\ref{conj:rank} is equivalent to the statement that~$x_n \to \infty$, where~$x_n \defeq - \log \frac{ \tilde{m}(n)}{m(n)}$. We argue that if~$x_n$ diverges sufficiently fast, then Conjecture~\ref{conj:transpositions} holds. The argument is similar to that used in the proof of Lemma~\ref{lemma:t-liminf}.
\begin{lemma}\label{lemma:rank-implies}
	There is a sequence~$(b_n)$ such that if~$x_n \ge \frac{b_n}{n+2} \binom{n+2}{\lfloor (n+2)/2\rfloor}$ for sufficiently large~$n$, then Conjecture~\ref{conj:transpositions} holds.
\end{lemma}
\begin{proof}
	Write
	\begin{equation*}
		f(n; \pi)\defeq
		\begin{cases}
			m(n, n/2; \pi) & \text{even~$n$,} \\
			m\left(n, \frac{n-1}{2}; \pi\right) + m\left(n, \frac{n+1}{2}; \pi\right) & \text{odd~$n$.}
		\end{cases}
	\end{equation*}
	There are~$\binom{n}{2}$ transpositions, and hence if Conjecuture~\ref{conj:rank} holds, then~$\frac{\binom{n}{2}f(n;\pi)}{m(n)} \to 0$ implies Conjecture~\ref{conj:transpositions}. We will show that this limit indeed exists and equals~$0$ if~$x_n \to\infty$ sufficiently fast.
	
	An application of~\eqref{eq:t-r-upper}, followed by taking logarithms, gives
	\begin{equation*}
		\log f(n; \pi) \le 1 + 2 \log(m(n-2)) - 2x_{n-2},
	\end{equation*}
	and so, using~\eqref{eq:matroids-upperbound} to bound~$\log m(n-2)$,
	\begin{multline*}
		\log \frac{\binom{n}{2} f(n; \pi)}{m(n)} \le 1 + \log \binom{n}{2} + \frac{4 + o(1)}{n-2}\binom{n-2}{\lfloor (n-2)/2\rfloor} - \frac{1}{n} \binom{n}{\lfloor n/2\rfloor} - 2x_{n-2} \\
			= \frac{o(1)}{n}\binom{n}{\lfloor n/2\rfloor} - 2x_{n-2}.
	\end{multline*}
	If~$x_n \ge \frac{1}{n+2} \binom{n+2}{\lfloor (n+2)/2\rfloor}$ for sufficiently large~$n$, then~$\frac{\binom{n}{2} f(n; \pi)}{m(n)} \to 0$.
\end{proof}

The proof of Lemma~\ref{lemma:rank-implies} shows that~$b_n = 1$ suffices. In fact, careful analysis of the $o(1)$-term that appears in~\eqref{eq:matroids-upperbound} (see e.g.\ \cite{BPvdP2015}), shows that one can take~$b_n = \Omega\left(\frac{\log^2 n}{n}\right)$ as well.

\subsection{An additional conjecture}

Our attempted resolution of Conjecture~\ref{conj:asymmetric} is thwarted by the matroids whose automorphism group is generated by a transposition. Even in the case of sparse paving matroids, for which we have been able to prove the conjecture, the bound on the number of matroids  whose automorphism group is generated by a single transposition is much weaker than the bound on the number of matroids that have an automorphism with larger support. The apparent difficulty in bounding the number of matroids whose automorphism group is generated by a single transposition leads us to making the following conjecture.

\begin{conjecture}\label{conj:aaa_symmetric}
	Asymptotically almost all symmetric matroids have an automorphism group that is generated by a transposition.
\end{conjecture}

A positive answer to Conjecture~\ref{conj:aaa_symmetric} would reflect the situation for graphs. It was shown by Erd\H{o}s and R\'enyi~\cite{ErdosRenyi1963} that asymptotically almost all graphs are asymmetric. From their proof it follows that the number of symmetric graphs is dominated by the number of graphs with automorphism group generated by a single transposition.

\bibliographystyle{alpha}
\bibliography{bib}

\end{document}